\documentclass[11pt,reqno]{amsart} 
\usepackage[utf8]{inputenc}
\usepackage{amsfonts,amsmath,amsthm,amssymb}
\usepackage{epstopdf}
\usepackage{subcaption}
\usepackage{thmtools}
\usepackage{thm-restate}

\newtheorem{theorem}{Theorem}
\newtheorem{corollary}{Corollary}
\newtheorem{lemma}{Lemma}
\newtheorem{definition}{Definition}
\newtheorem{remark}{Remark}
\newtheorem{example}{Example}

\usepackage{thmtools}
\usepackage{thm-restate}

\usepackage[colorlinks = true,
linkcolor = blue,
urlcolor  = blue,
citecolor = blue,
anchorcolor = blue]{hyperref}
\usepackage{cleveref}

\usepackage[utf8]{inputenc}
\usepackage{amsfonts,amsmath,amsthm,amssymb}
\usepackage[justification=centering]{caption}
\usepackage{color}
\usepackage{fullpage}
\usepackage{enumitem}   
\usepackage{hyperref}
\usepackage{mathtools}
\usepackage{xcolor}

\makeatletter
\newcommand{\addresseshere}{
  \enddoc@text\let\enddoc@text\relax
}
\makeatother

\newcommand*{\y}{\mathbf{y}}
\newcommand*{\x}{\mathbf{x}}

\newcommand*{\PS}{\mathsf{PS}}
\newcommand*{\PF}{\mathsf{PF}}
\newcommand*{\PA}{\mathsf{PA}}

\newcommand*{\Pref}[1]{\mathsf{Pref}_{#1}}

\newcommand*{\N}{\mathbb{N}}
\newcommand*{\Out}{\mathcal{O}}
\newcommand{\Sym}{\mathfrak{S}}

\usepackage[colorinlistoftodos]{todonotes}

\title{
Counting Parking Sequences and Parking Assortments Through Permutations
}

\author[Franks]{Spencer J. Franks}

\author[Harris]{Pamela E. Harris}
\thanks{P.~E.~Harris was supported through a Karen Uhlenbeck EDGE Fellowship.
}

\author[Harry]{Kimberly Harry}

\author[Kretschmann]{Jan Kretschmann}

\author[Vance]{Megan Vance}

\address[Franks, Harris, Harry, Kretschmann, Vance]{Department of Mathematical Sciences, University of Wisconsin-Milwaukee, Milwaukee, WI 53211}
\email{\textcolor{blue}{\href{mailto:sjfranks@uwm.edu}{sjfranks@uwm.edu}}, \textcolor{blue}{\href{mailto:peharris@uwm.edu}{peharris@uwm.edu}}, \textcolor{blue}{\href{mailto:kjharry@uwm.edu}{kjharry@uwm.edu}}, \textcolor{blue}{\href{mailto:kretsc23@uwm.edu}{kretsc23@uwm.edu}}, \textcolor{blue}{\href{mailto:mmvance@uwm.edu}{mmvance@uwm.edu}}}

\begin{document}

\begin{abstract}
Parking sequences (a generalization of parking functions) are defined by specifying car lengths and requiring that a car attempts to park in the first available spot after its preference. If it does not fit there, then a collision occurs and the car fails to park. In contrast, parking assortments generalize parking sequences (and parking functions) by allowing cars (also of assorted lengths) to seek forward from their preference to identify a set of contiguous unoccupied spots in which they fit. 
We consider both parking sequences and parking assortments and establish that the
number of preferences resulting in a fixed parking order $\sigma$ is related to the lengths of cars indexed by certain subsequences in $\sigma$.
The sum of these numbers over all parking orders (i.e.~permutations of $[n]$) yields new formulas for the total number of parking sequences and of parking assortments. 
\end{abstract}

\maketitle

\section{Introduction}
Throughout, we let 
$n\in\N=\{1,2,3,\ldots\}$ and $[n]=\{1,2,3,\ldots,n\}$. 
\emph{Parking sequences}, as introduced by Ehrenborg and Happ~\cite{ehrenborg2016parking}, are defined as follows. 
Suppose there are $n$ cars $1, 2,\ldots, n$ of lengths $y_1,y_2, \ldots , y_n\in\N$, respectively. 
Let $m = \sum_{i=1}^{n}y_i$ be the number of parking spots on a one-way street.
Sequentially label parking spots $1,2,3,\ldots,m$ increasingly along the direction of a one-way street. 
We let $x_i\in [m]$ denote the preferred spot of car $i$, for all $i\in[n]$, and we say  $\x=(x_1, x_2, \ldots, x_n)$ is the \emph{preference list} for the cars with lengths $\y=(y_1,y_2,\ldots,y_n)$.
The cars enter the one-way street from the left in the order $1, 2,\ldots,n$; and car $i$ seeks the first empty spot $j\geq x_i$.
If all of the spots $j,j+1,\ldots,j+y_i-1$ are empty, then car $i$ parks there. 
If spot $j$ is empty and at least one of the spots $j+1,j+2,\ldots,j+y_i-1$ is occupied, then there is a collision; and car $i$ fails to park.
If all cars park successfully under the preference list $\x$, then $\x$ is a \emph{parking sequence} for $\y$.
We denote the set of parking sequences for $\y$ by $\PS_n(\y)$.
Figure \ref{fig:ps example} illustrates examples $\x$ which are and are not parking sequences for $\y=(1,2,1)$.

\begin{figure}[!h]
    \centering
    \begin{tikzpicture}
\node at (2,-1) {$\x=(3,1,4)$};
    \draw[step=1cm,gray,very thin] (0,0) grid (4,1);
    \draw[fill=gray!50] (2.1,0.1) rectangle (2.9,.9);
    \node at (2.5,.5) {$1$};
    \draw[fill=gray!50] (.1,0.1) rectangle (1.9,.9);
    \node at (1,.5) {$2$};
    \draw[fill=gray!50] (3.1,0.1) rectangle (3.9,.9);
    \node at (3.5,.5) {$3$};
    \node at (.5,-.25) {$1$};
    \node at (1.5,-.25) {$2$};
    \node at (2.5,-.25) {$3$};
    \node at (3.5,-.25) {$4$};
    \end{tikzpicture}
    \qquad\qquad\qquad
    \begin{tikzpicture}
    \node at (2,-1) {$\x=(2,1,1)$};
    \draw[step=1cm,gray,very thin] (0,0) grid (4,1);
     \draw[fill=gray!50] (1.1,0.1) rectangle (1.9,.9);
    \node at (1.5,.5) {$1$};
    \draw[fill=gray!50] (-2,1.1) rectangle (-.1,1.9);
    \node at (-1,1.5) {$2$};
    \draw[->,red, thick, out=0,in=90](0,1.5) to (.5,.5);
    \draw[->,red, thick, out=90,in=180](.5,.5) to (2.5,1.5);
    \node at (.5,-.25) {$1$};
     \node at (1.5,-.25) {$2$};
     \node at (2.5,-.25) {$3$};
     \node at (3.5,-.25) {$4$};
    \end{tikzpicture}
    \caption{Note $\x=(3,1,4)$ is a parking sequence for  $\y=(1,2,1)$ in which 
car $1$ of length 1 parks in spot 3, 
car $2$ of length 2 parks in spots 1 and 2, and 
car $3$ of length 1 parks in spot 4. 
On the other hand, $\x=(2,1,1)$ is not a parking sequence for $\y$, since car $2$ collides with car $1$ when attempting to park.
}
    \label{fig:ps example}
\end{figure}
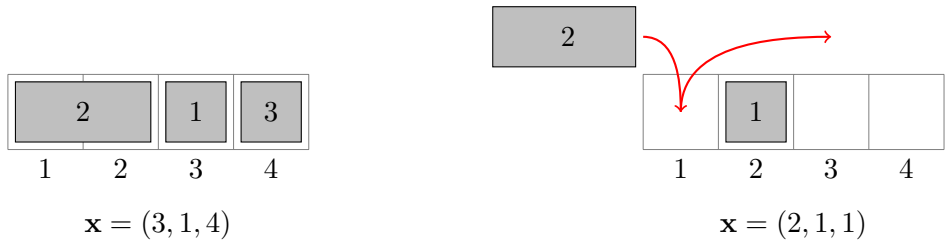

We remark that the set of parking sequences for $n$ cars each with unit length (the case where $\y=(1,1,\ldots,1)\in\N^n$) is precisely the set of (classical) parking functions, which we denote by $\PF_n$. 
Thus, parking sequences are a generalization of parking functions. 
Ehrenborg and Happ established the following 
\cite[Theorem 1.3]{ehrenborg2016parking}: The number of parking sequences for cars with lengths $\y=(y_1,y_2,\ldots,y_n)$ is given by the product
\begin{align}
    |\PS_n(\y)|&=(y_1+n)(y_1+y_2+n-1)\cdots(y_1+\cdots+y_{n-1}+2).\label{ps count}
\end{align} 
Ehrenborg and Happ's proof of \eqref{ps count}  constructed a ``circular street'' on which the cars park, an argument used by Pollack (see, \cite{ballotsandtrees}) to establish that $|\PF_n|=(n+1)^{n-1}$.

Given a parking sequence $\x$ for $\y$, the result of the parking experiment yields a permutation $\sigma=\sigma_1\sigma_2\cdots\sigma_n$ of $[n]$, written in one-line notation, which denotes the order in which the cars park on the street. 
Note that $\sigma$ corresponds to the order in which the cars park, not the order in which they arrive. 
Namely, for each $j\in[n]$, $\sigma_j=i$ denotes that car $i$ is the $j$th car parked on the street. 
In this work, we are interested in determining an alternative way of counting the number of parking sequences for $\y$, by keeping track of those that park the cars in the order $\sigma$. 
To this effect, we let $\Sym_n$ denote the set of permutations on $[n]$ and for a fixed $\y$ we define the outcome map $\Out_{\PS_n(\y)}: \PS_n(\y) \rightarrow \mathfrak{S}_n$ by $\Out_{\PS_n(\y)}(\x)=\sigma=\sigma_1\sigma_2\cdots\sigma_n$
 and, given $\sigma\in\mathfrak{S}_n$, we study the fibers of the outcome map:
\[\Out^{-1}_{\PS_n(\y)}(\sigma)=
\{\x\in\PS_n(\y): \Out_{\PS_n(\y)}(\x)=\sigma\}.
\]
Our first main result, proved in Section \ref{sec:parking sequneces}, establishes the following.

\begin{restatable}{thm}{restatethm}
\label{theorem:parking sequences}
    Fix $\y=(y_1,y_2,\ldots,y_n)\in\N^n$ and  $\sigma=\sigma_1\sigma_2\cdots\sigma_n\in\mathfrak{S}_n$. Then 
    \[|\Out_{\PS_n(\y)}^{-1}(\sigma)|=\prod_{i=1}^n\left(1+\sum_{k\in L(\y,\sigma_i)}y_k\right),\]
    where 
$L(\y,\sigma_i)=\emptyset$ if $i=1$ or if $\sigma_{i-1}>\sigma_i$, otherwise 
$L(\y, \sigma_i)=\{\sigma_t,\sigma_{t+1},\ldots,\sigma_{i-1}\}$ with $\sigma_t\sigma_{t+1}\ldots \sigma_i$ being the longest subsequence of $\sigma$ such that  $\sigma_k< \sigma_i$ for all $t\le k< i$.
\end{restatable}

The specialization of $\y=(1,1,\ldots,1)\in\N^n$ in Theorem \ref{theorem:parking sequences} recovers \cite[Proposition 3.1]{knaple}:
    Let $\sigma = \sigma_1 \sigma_2 \ldots \sigma_n$
    be a permutation in $\mathfrak{S}_n$.
    Then $|\Out_{\PF_n}^{-1}(\sigma)|= \prod_{i=1}^n \ell(i; \sigma)$, 
    where $\ell(i;\sigma)$ is the length of the longest subsequence $\sigma_j \cdots \sigma_i$ of $\sigma$ such that $\sigma_t\leq \sigma_i$ for all $j\leq t \leq i$.

The following result gives an alternate new formula for the number of parking sequences for a fixed $\y$ as a sum over permutations.

\begin{corollary}\label{cor:count ps}
Fix $\y=(y_1,y_2,\ldots,y_n)\in\N^n$ and let
$L(\y,\sigma_i)$ be defined as in Theorem \ref{theorem:parking assortments}. Then 
     \begin{align}\label{eq:gen ps}
    |\PS_n(\y)|&=\sum_{\sigma\in\mathfrak{S}_n}|\Out_{\PS_n(\y)}^{-1}(\sigma)|=\sum_{\sigma\in\mathfrak{S}_n}\left(\prod_{i=1}^n\left(1+\sum_{k\in L(\y,\sigma_i)}y_k\right)\right).\end{align}
     
\end{corollary}

Next, we consider a generalization of parking sequences, known as parking assortments, for which we provide analogous results to those in Theorem \ref{theorem:parking sequences} and Corollary \ref{cor:count ps}. 

\emph{Parking assortments}, as introduced by Chen, Harris, Mart\'inez Mori, Pab\'on-Cancel, and Sargent~\cite{parkingassortments},  are defined as follows. 
As before, we fix $\y=(y_1,y_2,\ldots,y_n)\in\mathbb{N}^n$ to denote the car lengths and we suppose the cars have preferences $\x=(x_1,x_2,\ldots,x_n)\in[m]^n$, where $m=\sum_{i=1}^n y_i$. 
For $i\in[n]$,
car $i$ enters the one-way street from the left and drives to its preferred spot $x_i$. 
If spots $x_i,x_i+1,\ldots, x_i+y_i-1$ are unoccupied, then it parks. 
Otherwise, car $i$ proceeds down the one-way street, parking in the first contiguous unoccupied $y_i$ parking spots it encounters.
Throughout the paper, we refer to the spot(s) car $i$ parks in simply by the leftmost spot it occupies.
If no such parking spot(s) are found, then we say parking fails. 
If $\x$ is a preference list allowing all cars to park on the $m$ spots on the street, then we say that $\x$ is a \emph{parking assortment} for $\y$. 
We let $\PA_n(\y)$ denote the set of all parking assortments for $\y$. 
Observe that all parking sequences are assortments, i.e.~$\PS_n(\y)\subseteq \PA_n(\y)$. 
If $n=2$, then $\PS_2(\y)=\PA_2(\y)$ (Lemma \ref{lemma:equal}). 
This result relies on the fact that in either a parking sequence or parking assortment, there exists at least one car preferring the first spot. Otherwise, by pigeonhole principle, the cars are unable to park.
However, we have ample evidence that for $n\geq 3$ there are parking assortments that are not parking sequences. 
We illustrate such an example in Figure \ref{fig:pa not ps}.

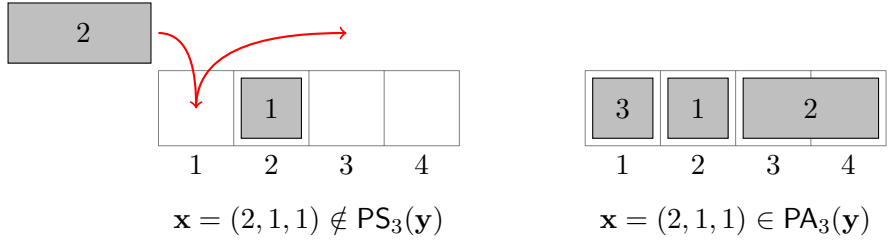
\begin{figure}[h]
    \centering
    \begin{tikzpicture}
    \node at (2,-1) {$\x=(2,1,1)\notin \PS_{3}(\y)$};
    \draw[step=1cm,gray,very thin] (0,0) grid (4,1);
     \draw[fill=gray!50] (1.1,0.1) rectangle (1.9,.9);
    \node at (1.5,.5) {$1$};
    \draw[fill=gray!50] (-2,1.1) rectangle (-.1,1.9);
    \node at (-1,1.5) {$2$};
    \draw[->,red, thick, out=0,in=90](0,1.5) to (.5,.5);
    \draw[->,red, thick, out=90,in=180](.5,.5) to (2.5,1.5);
    \node at (.5,-.25) {$1$};
     \node at (1.5,-.25) {$2$};
     \node at (2.5,-.25) {$3$};
     \node at (3.5,-.25) {$4$};
    \end{tikzpicture}
    \qquad\qquad
\begin{tikzpicture}
    \node at (2,-1) {$\x=(2,1,1)\in \PA_{3}(\y)$};
    \draw[step=1cm,gray,very thin] (0,0) grid (4,1);
     \draw[fill=gray!50] (1.1,0.1) rectangle (1.9,.9);
    \node at (1.5,.5) {$1$};
    \draw[fill=gray!50] (2.1,.1) rectangle (3.9,.9);
    \node at (3,.5) {$2$};
     \draw[fill=gray!50] (.1,0.1) rectangle (.9,.9);
     \node at (.5,.5) {$3$};
    \node at (.5,-.25) {$1$};
     \node at (1.5,-.25) {$2$};
     \node at (2.5,-.25) {$3$};
     \node at (3.5,-.25) {$4$};
    \end{tikzpicture}
    \caption{Let $\y=(1,2,1)$. In Figure \ref{fig:ps example} we showed $(2,1,1)\notin\PS_3(\y)$. However, under the parking assortment rule: car $1$ parks in spot 2. Car $2$ attempts to park in spot 1, unable to fit there, it continues down the street parking in spot 3 (occupying spots 3 and 4).
    Car $3$ finds spot 1 available, able to fit, it parks there.
    Thus, $\x\in \PA_3(\y)$.}
    \label{fig:pa not ps}
\end{figure}

We also consider the analogous study of the set of parking assortments resulting in a particular parking order. 
To make this precise we fix $\y$, and define the outcome map $\Out_{\PA_n(\y)}: \PA_n(\y) \rightarrow \mathfrak{S}_n$ by $\Out_{\PA_n(\y)}(\x)=\sigma=\sigma_1\sigma_2\cdots\sigma_n$, where $\sigma_j=i$ denotes that car $i$ is the $j$th car parked on the street. 
For a fixed $\sigma\in\mathfrak{S}_n$ we study the fibers of the outcome map:
\[
\Out^{-1}_{\PA_n(\y)}(\sigma)=
\{\x\in\PA_n(\y): \Out_{\PA_n(\y)}(\x)=\sigma\}.
\]
In Section \ref{sec:enum parking assortments}, we fix $\y\in\N^n$ and give the cardinality of $\Out^{-1}_{\PA_n(\y)}(\sigma)$ for any $\sigma\in\Sym_n$ (Theorem \ref{theorem:parking assortments}). 
Using this result, we establish a formula for the cardinality of $\PA_n(\y)$ as a sum over permutations (Corollary \ref{cor:count pa}). Note that Theorem \ref{theorem:parking assortments} and Corollary \ref{cor:count pa} are generalizations of Theorem \ref{theorem:parking sequences} and Corollary \ref{cor:count ps}, respectively.

We conclude with some applications of our enumerative results. In particular, for certain car lengths $\y\in\N^n$, we give the cardinality of the sets $\Out_{\PA_n}^{-1}(123\cdots n)$  as $n$ increases, where $123\cdots n$ is the identity permutation in $\mathfrak{S}_n$. Note that in this case $\Out_{\PS_n(\y)}^{-1}(123\cdots n)=\Out_{\PA_n(\y)}^{-1}(123\cdots n)$.

\begin{enumerate}
    \item If $\y=(1,2,3,\ldots,n)$, then the cardinalities of $\Out_{\PA_n}^{-1}(123\cdots n)$  as $n$ increases is \[1, 2, 8, 56, 616, 9856, 216832, 6288128,\ldots.\] This sequence agrees with OEIS \href{https://oeis.org/A128814}{A128814}: the partial products of the Lazy Caterer's Sequence. We recall that the Lazy Caterer's sequence (OEIS \href{https://oeis.org/A000124}{A00124}) gives the ``maximal number of pieces formed when slicing a pancake with $n$ cuts.''
    
    \item If $\y=(1,1,2,3,5,8,\ldots)$ consists of the first $n$ Fibonacci numbers,  then the cardinality of $\Out_{\PA_n}^{-1}(123\cdots n)$  as $n$ increases is \[1, 2, 6, 30, 240, 3120, 65520, 2227680,\ldots.\] This sequence agrees with OEIS \href{https://oeis.org/A003266}{A003266}: The product of the first $n+1$ nonzero Fibonacci numbers.
    \item If $\y=(1,3,9,28,90,\ldots)$, consists of the first $n+1$ terms of the sequence enumerating standard tableaux of shape $(n+1,n-1)$ (starting at index 2, OEIS \href{https://oeis.org/A071724}{A071724}), namely
    $y_i=\frac{3}{i+3}\binom{2(i+1)}{i}$ for $i\in [n+1]$,  then the cardinality of $\Out_{\PA_n}^{-1}(123\cdots n)$  as $n$ increases is \[1,2,10,140,5880,776160,\ldots.\] This sequence agrees with OEIS \href{https://oeis.org/A003046}{A003046}: The partial products of the first $n+1$ Catalan numbers.
\end{enumerate}

\section{Enumerating parking sequences}\label{sec:parking sequneces}
In this section, we prove Theorem \ref{theorem:parking sequences}. 
We begin by stating the following definition.
\begin{definition}
For each $i\in[n]$ and $\sigma=\sigma_1\sigma_2\cdots\sigma_n\in\mathfrak{S}_n$, we let $\Pref{\PS_n}(\sigma_i)$ be the set of possible preferences for car $\sigma_{i}$ so that it is the $i$th car to park on the street when using the parking sequence parking rule. We let $|\Pref{\PS_n}(\sigma_i)|$ denote the cardinality of the set.
\end{definition}

To begin, we present an example in which we compute the set of preferences for cars (of certain lengths) parking in a predetermined order.

\begin{example}\label{ex:1}
Let $\y=(1,6,5,5,3,2,2)$ and consider the parking order described by the permutation $\sigma=2457361$. 
We consider cars as they parked on the street from left to right in order to
determine the preferences for each car so that the parking process result in the cars parking in the order $\sigma$:
\begin{itemize}
    \item  Car $2$ is parked  first in the sequence of cars. Since there are no cars parked to the left of car $2$, there is only 1 spot car $2$ could have preferred, precisely where it is parked. Hence, $\Pref{\PS_7}(\sigma_1)=\{1\}$.

\item Car $4$ is parked second in the sequence of cars. 
Since car $2$ parked to the left of and earlier than car $4$, car $4$ could have preferred the spot it parked in or any of the spots occupied by car $2$.
Thus, $\Pref{\PS_7}(\sigma_2)=\{1,2,3,4,5,6,7\}$.

\item Car $5$ is parked third in the sequence of cars. 
Since car $2$ and car $4$ parked to the left of and earlier than car $5$, car $5$ could have preferred the spot it parked in or any of the spots occupied by car $2$ or by car $4$. Thus, 
$\Pref{\PS_7}(\sigma_3)=\{1,2,\ldots, 11,12\}$.

\item Car $7$ is parked fourth in the sequence of cars. 
Since cars $2$, $4$, and $5$ parked to the left of and earlier than car $7$, car $7$ could have preferred the spot it parked in or any of the spots occupied by cars $2$, $4$, or $5$.
Thus, 
$\Pref{\PS_7}(\sigma_4)=\{1,2,\ldots,14,15\}$.

\item Car $3$ is parked fifth in the sequence of cars. 
Since car $7$ parked to the left of car  $3$ but entered the street after car $3$, car $3$ could not have preferred any spots to the left of where car $3$ parked. 
Thus,
$\Pref{\PS_7}(\sigma_5)=\{17\}$.

\item Car $6$ is parked sixth in the sequence of cars. 
Since car $3$ parked to the left of and earlier than car $6$, car $6$ could have preferred the spot it parked in or any of the spots occupied by car $3$.
Moreover, as the next car to the left of car $3$ is car $7$, which arrived after car $6$, then car $6$ could not have preferred any of the spots car $7$ parks in or those before car $7$.
Thus, $\Pref{\PS_7}(\sigma_6)=\{17, 18, \ldots, 22\}$.

\item Car $1$ is parked seventh in the sequence of cars. 
Since car $6$ parked to the left of car $1$ but entered the street after car $1$, car $1$ could not have preferred any spots to the left of where car $1$ parked. 
Thus, $\Pref{\PS_7}(\sigma_7)=\{24\}$.
\end{itemize}

These computations show that
$\Out_{\PS_7(\y)}^{-1}(\sigma)=
\Pref{\PS_7}(\sigma_1)\times\Pref{\PS_7}(\sigma_2)\times\Pref{\PS_7}(\sigma_3)\times\Pref{\PS_7}
(\sigma_4)\times\Pref{\PS_7}(\sigma_5)\times\Pref{\PS_7}
(\sigma_6)\times\Pref{\PS_7}(\sigma_7)$
and hence
\begin{align}\label{eq:example}
|\Out_{\PS_7(\y)}^{-1}(\sigma)|=\prod_{i=1}^7|\Pref{\PS_7}(\sigma_i)|=1 \cdot 7\cdot 12\cdot 15\cdot 1\cdot 6\cdot 1=7560.
\end{align}
\end{example}

As Example~\ref{ex:1} illustrates, in computing $|\Pref{\PS_n}(\sigma_i)|$, it is important to know which cars parked to the left of car $\sigma_i$ and when they arrived in the queue, as this affects the possible preferences car $\sigma_i$ can have. 
This motivates the following.

\begin{definition}
Fix $\y\in\N^n$ and $\sigma=\sigma_1\sigma_2\cdots\sigma_n\in\mathfrak{S}_n$. 
For $i\in[n]$, 
let $\sigma_t\sigma_{t+1}\ldots \sigma_i$ be the longest subsequence of $\sigma$ such that $\sigma_k< \sigma_i$ for all $t\le k< i$.
\begin{enumerate}[leftmargin=.2in]
    \item If $i=1$, then define $L(\y,\sigma_i)=\emptyset$,
    \item if $\sigma_{i-1}>\sigma_i$, then define $L(\y,\sigma_i)=\emptyset$, and
    \item otherwise define $L(\y, \sigma_i)=\{\sigma_t,\sigma_{t+1},\ldots,\sigma_{i-1}\}$.
\end{enumerate}
\end{definition}

Note that albeit technical, the definition of $L(\y;\sigma_i)$ simply keeps track of the cars parked consecutively  left of $\sigma_i$ which arrived before $\sigma_i$.

\begin{lemma}\label{lemma:pref for each car}
Fix $\y\in\N^n$ and $\sigma=\sigma_1\sigma_2\cdots\sigma_n\in\mathfrak{S}_n$. If $i\in[n]$, then
\[|\Pref{\PS_n}(\sigma_i)|=1+\displaystyle\sum_{k\in L(\y,\sigma_i)}y_k.\]
\end{lemma}

\begin{proof}
For any $j\in[n]$, recall that $\sigma_j=i$ denotes that car $i$ is the $j$th car parked on the street.
The only possible preferences for car $i$ is the initial spot it parks in or any of the spots contiguously occupied by cars parked to the left of car $i$ which arrived before it.
Such cars are those in the set $L(\y,\sigma_j)$.
Note that if car $i$ preferred any earlier spot, then there would be a collision or it would park elsewhere on the street, contradicting that car $i$ was the $j$th car on the street.
Therefore, this establishes that 
$|\Pref{\PS_n}(\sigma_j)|=1+\sum_{k\in L(\y,\sigma_j)}y_k$, as claimed.
\end{proof}

Using Lemma \ref{lemma:pref for each car}, we confirm the computation in \eqref{eq:example} next.
\begin{example}
As in Example \ref{ex:1}, let 
$\y=(1,6,5,5,3,2,2)\in\N^7$ and
and $\sigma=2457361\in\mathfrak{S}_7$. Then 
\begin{itemize}
    \item $L(\y, \sigma_1)=\emptyset$ and $|\Pref{\PS_7}(\sigma_1)|=
    1+0=1$,
    \item $L(\y, \sigma_2)=\{2\}$ and $|\Pref{\PS_7}(\sigma_2)|=
    1+y_2
    =1+6=7$,
    \item $L(\y, \sigma_3)=\{2,4\}$ and $|\Pref{\PS_7}(\sigma_3)|=
    1+y_2+y_4=
    1+6+5=12$,
    \item $L(\y, \sigma_4)=\{2,4,5\}$ and $|\Pref{\PS_7}(\sigma_4)|=
    1+y_2+y_4+y_5
    =1+6+5+3=15$,
    \item $L(\y, \sigma_5)=\emptyset$ and $|\Pref{\PS_7}(\sigma_5)|
    =1+0=1$,
    \item $L(\y, \sigma_6)=\{3\}$ and $|\Pref{\PS_7}(\sigma_6)|=
    1+y_3=
    1+5=6$, and
    \item $L(\y, \sigma_7)=\emptyset$ and $|\Pref{\PS_7}(\sigma_7)|=
    1+0=1$.
\end{itemize}

This confirms that $|\Out^{-1}_{\PS_7(\y)}(\sigma)|=7560$ as computed in \eqref{eq:example}.
\end{example}

For convenience we restate our main result.

\restatethm*

\begin{proof}
As the preferences for each car are independent, we know that 
$|\Out_{\PS_n(\y)}^{-1}(\sigma)|=\prod_{i=1}^n|\Pref{\PS_n}(\sigma_i)|$. Then by Lemma \ref{lemma:pref for each car}, we know that for each $i\in[n]$, $|\Pref{\PS_n}(\sigma_i)|=1+\sum_{k\in L(\y,\sigma_i)}y_k$, from which the result follows. 
\end{proof}

\section{Enumerating parking assortments}\label{sec:enum parking assortments}

In this section, we prove Theorem \ref{theorem:parking assortments}. 
We begin by establishing the following initial result.
\begin{lemma}\label{lemma:equal}
If $\y=(y_1,y_2)\in\N^2$, then $\PS_2(\y)=\PA_2(\y)$.
\end{lemma}
\begin{proof}
    We know that $\PS_n(\y)\subseteq\PA_n(\y)$ for all $n$.
    It suffices to show that if $\x\in\PA_2(\y)$, then $\x\in\PS_2(\y)$. We establish this next.
    
    Assume $\x=(x_1,x_2)\in[y_1+y_2]^2 \in \PA_2(\y)$. 
    Hence, at least one of the two cars must prefer parking spot~1.
    Then there are two possibilities:
    
    \begin{itemize}[leftmargin=0in]
        \item[] \textbf{Case 1:} $x_1=1$ and $x_2 \leq y_1+1$.
        In this case, car 1 parks in spot 1 and car 2 begins looking for an open spot early enough so that car 2 fits on the street. 
        Hence, $\x \in \PS_2(\y)$.\\
    \item[] \textbf{Case 2:} $x_1=y_2+1$ and $x_2=1$. 
    In this case, car 1 does not park in spot 1. 
    So, in order for $\x$ to be a parking sequence, car 1 must park in the $y_1$ rightmost spots on the street (i.e.~prefers spot $y_2+1$),
    otherwise car 1 would break up the street, leaving open spots to the left and to the right of car 1. 
    Also, since car 1 does not prefer spot 1, car 2 must prefer spot 1. Hence, $\x \in \PS_2(\y)$.\qedhere
    \end{itemize}
\end{proof}

The following definitions set some needed notation for our enumerative results.

\begin{definition}
For each $i\in[n]$ and $\sigma=\sigma_1\sigma_2\cdots\sigma_n\in\mathfrak{S}_n$, we let $\Pref{\PA_n}(\sigma_i)$ be the set of possible preferences for car $\sigma_{i}$ so that it is the $i$th car to park on the street when using the parking assortment parking rule. We let $|\Pref{\PA_n}(\sigma_i)|$ denote the cardinality of the set.
\end{definition}

\begin{definition}\label{def:complicated}
Let $\sigma=\sigma_1\sigma_2\cdots\sigma_n\in\mathfrak{S}_n$. 
Fix $i\in[n]$ and partition the subword $T(\sigma_i)\coloneqq \sigma_1\sigma_2\cdots\sigma_{i-2}\sigma_{i-1}$ as follows:
\begin{enumerate}[leftmargin=.2in]
    \item If $i=1$, then $T(\sigma_1)=\emptyset$.

\item If $\sigma_{i-1}>\sigma_i$, then 

\begin{align}
T(\sigma_i)=\sigma_1\sigma_2\cdots\sigma_{i-2}\sigma_{i-1}=
\begin{cases}
  \beta_\ell\alpha_\ell\cdots\beta_2\alpha_2 \beta_1 \alpha_1&\mbox{if $\sigma_1<\sigma_i$}\\
  \alpha_{\ell+1}\beta_\ell\alpha_\ell\cdots\beta_2\alpha_2 \beta_1 \alpha_1&\mbox{if $\sigma_1>\sigma_i$}
\end{cases}\label{eq:left is bigger}
\end{align}
where 
$\alpha_1$ is the longest contiguous subword consisting of $\sigma_j>\sigma_i$ and  $\beta_1$ is the longest contiguous subword consisting of $\sigma_j<\sigma_i$.  
We iterate in this way until we consider $\sigma_1$ arriving at one of the two cases in \eqref{eq:left is bigger}.

    \item If $\sigma_{i-1}<\sigma_i$, then 
\begin{align}
T(\sigma_i)=\sigma_1\sigma_2\cdots\sigma_{i-2}\sigma_{i-1}=\begin{cases}
  \alpha_\ell\beta_\ell\cdots\alpha_2 \beta_2 \alpha_1\beta_1&\mbox{if $\sigma_1>\sigma_i$}\\
  \beta_{\ell+1}\alpha_\ell\beta_\ell\cdots\alpha_2 \beta_2\alpha_1\beta_1&\mbox{if $\sigma_1<\sigma_i$}
\end{cases}\label{eq:left is smaller}
\end{align}
where 
$\beta_1$ is the longest contiguous subword consisting of $\sigma_j<\sigma_i$ and $\alpha_1$ is the longest contiguous subword consisting of $\sigma_j>\sigma_i$. 
We iterate in this way until we consider $\sigma_1$ arriving at one of the two cases in \eqref{eq:left is smaller}. 
\end{enumerate}
\end{definition}
In Definition \ref{def:complicated}, we use the letters $\alpha$ (or $\beta$) to identify cars parked to the left of a particular car which arrived ``after'' (or ``before'') it. 
We illustrate Definition \ref{def:complicated} next.
\begin{example}\label{ex:partitioning sigma}
    If $\sigma=4123$, then
    $T(\sigma_1)=\emptyset$, $ T(\sigma_ 2)=\underbrace{4}_{\alpha_1}$, $T(\sigma_ 3)=\underbrace{4}_{\alpha_1}\underbrace{1}_{\beta_1}$, $ T(\sigma_ 4)=\underbrace{4}_{\alpha_1}\underbrace{12}_{\beta_1}$.
\end{example}

We soon show that Definition \ref{def:complicated} encapsulates all of the cases affecting the preferences for every car.
Moreover, in what follows, we abuse notation by thinking of $\alpha$'s and $\beta$'s both as subwords and as sets consisting of the values making up each respective subword.

\begin{theorem}\label{thm:assortments}
Let $\sigma=\sigma_1\sigma_2\cdots\sigma_n\in\mathfrak{S}_n$ and $\y=(y_1,y_2,\ldots,y_n)\in\mathbb{N}^n$.
Fix $i\in[n]$ and partition $T(\sigma_i)$ as in Definition \ref{def:complicated}. 
Then $\Pref{\PA_n}(\sigma_i)$ has the following cardinalities:
\begin{enumerate}[leftmargin=.2in]
    \item if $i=1$ or $\sigma_{i-1}>\sigma_i$, then $|\Pref{\PA_n}(\sigma_i)|=1$;
    \item if $T(\sigma_i)=\beta_1$, then $|\Pref{\PA_n}(\sigma_i)|=1+ \displaystyle\sum_{\sigma_k\in\beta_1}y_{\sigma_k}$;
    \item otherwise 
    \[
      |\Pref{\PA_n}(\sigma_i)|=\begin{cases}
      1+\displaystyle\sum_{k=1}^{i-1}y_{\sigma_k}&\qquad\mbox{if $m(i)$ does not exist}\\
      \displaystyle\sum_{\sigma_k\in \beta_{m(i)}\alpha_{m(i)-1}\beta_{m(i)-1}\cdots \alpha_1\beta_1\sigma_i}\hspace{-.5in}y_{\sigma_k}&\qquad\mbox{if $m(i)$ exists}
    \end{cases}
    \]
    where \[m(i)=\min\left\{1\leq j\leq \ell: \sum_{\sigma_k\in \alpha_j}y_{\sigma_k}\geq y_{\sigma_{i}}\right\}.\]
    \end{enumerate}
\end{theorem}

\begin{proof}
We proceed by proving each case independently.
\begin{itemize}[leftmargin=0in]
    \item[]  \textbf{Case 1:} If $i=1$, then $\sigma_i=\sigma_1$ is the first car parked on the street, which implies that it must have preferred the first parking spot on the street. Hence $|\Pref{\PA_n}(\sigma_1)|=1$, as claimed. 
If $\sigma_{i-1}>\sigma_i$, this means that the car parked immediately to the left of $\sigma_i$ arrived after $\sigma_i$. 
Hence car $\sigma_i$ can only prefer the spot it parked in, as otherwise it would have parked elsewhere. 
This implies $|\Pref{\PA_n}(\sigma_i)|=1$, as claimed.\\

\item[] \textbf{Case 2:} If $T(\sigma_i)=\beta_1$, then $\sigma_j<\sigma_i$ for all $j\in[i-1]$. 
Thus all of the cars parked left of $\sigma_i$ arrived and parked before $\sigma_i$. 
Hence $\sigma_i$ could prefer all of the spots cars $\sigma_1,\sigma_2,\ldots,\sigma_{i-1}$ occupy, as well as the spot in which $\sigma_i$ ultimately parks. 
This implies $|\Pref{\PA_n}(\sigma_i)|=1+\sum_{\sigma_j\in\beta_1}y_{\sigma_j}$.\\

\item[]\textbf{Case 3:} Note that $\sigma_{i-1}<\sigma_{i}$ (as otherwise this would be Case 1).
Furthermore, we can assume that $\alpha_1$ exists (as otherwise this would be Case 2). 
Hence, by Definition \ref{def:complicated}, we have
\[T(\sigma_i)=\sigma_1\sigma_2\cdots\sigma_{i-2}\sigma_{i-1}=\begin{cases}
  \alpha_\ell\beta_\ell\cdots\alpha_2 \beta_2 \alpha_1\beta_1&\mbox{if $\sigma_1>\sigma_i$}\\
  \beta_{\ell+1}\alpha_\ell\beta_\ell\cdots\alpha_2 \beta_2\alpha_1\beta_1&\mbox{if $\sigma_1<\sigma_i$}
\end{cases}\]
where 
$\beta_1$ is the longest contiguous subword consisting of $\sigma_j<\sigma_i$ and $\alpha_1$ is the longest contiguous subword consisting of $\sigma_j>\sigma_i$. 
In either case, we note that by definition, each $\alpha_j$ denotes a set of cars parking contiguously on the street, arriving in the queue after car $\sigma_i$ and parking to the left of car $\sigma_i$. 
The cars in the subwords $\alpha_j$ (for $1\leq j\leq \ell$) create gaps in the street which $\sigma_i$ could potentially park in if they happen to be large enough.

That is, for any $j\in [\ell]$, if $\sum_{\sigma_k\in\alpha_j}y_{\sigma_k}\geq y_{\sigma_i}$, then $\sigma_i$ 
\begin{itemize}
    \item cannot prefer all of the spots occupied by the cars in $\alpha_j$ and
    \item cannot prefer any spots to the left of the spots occupied by the cars in $\alpha_j$,
\end{itemize}
since then $\sigma_i$
would park either before or within the spots occupied by the cars in $\alpha_j$. 
Both cases contradict the fact that $\sigma_i$
is the $i$th car parked on the street.

In fact, the only parking spots car $\sigma_i$ could prefer are 
\begin{itemize}
\item the spots occupied by the cars in $\beta_{m(i)}\alpha_{m(i)-1}\beta_{m(i)-1}\cdots\alpha_1\beta_1$,
    \item the right-most $y_{\sigma_i}-1$ spots occupied by the cars in $\alpha_{m(i)}$, or
    \item the spot $\sigma_i$ parks in.   
\end{itemize}
Note that this exhausts all of the possible preferences for $\sigma_i$, as $\alpha_{m(i)}$ (by definition) is the closest gap in which $\sigma_i$ could park.
Thus, the number of spots that car $\sigma_i$ can prefer is 
\begin{align*}
|\Pref{\PA_n}(\sigma_i)|&=1+(y_{\sigma_i}-1)+\sum_{\sigma_k\in\beta_{m(i)}\alpha_{m(i)-1}\beta_{m(i)-1}\cdots\alpha_1\beta_1}y_{\sigma_k}
=\sum_{\sigma_k\in\beta_{m(i)}\alpha_{m(i)-1}\beta_{m(i)-1}\cdots\alpha_1\beta_1\sigma_i}y_{\sigma_k}
\end{align*}
as claimed.\qedhere
\end{itemize}
\end{proof}

We can now formally state and prove the analogous result to Theorem \ref{theorem:parking sequences} for parking assortments.
\begin{theorem}\label{theorem:parking assortments}
    Fix $\y=(y_1,y_2,\ldots,y_n)\in\N^n$ and let $\sigma=\sigma_1\sigma_2\cdots\sigma_n\in\mathfrak{S}_n$. Then 
    \[|\Out_{\PA_n(\y)}^{-1}(\sigma)|=\prod_{i=1}^n|\Pref{\PA_n}(\sigma_i)|,\]
    where 
    \begin{equation}
    |\Pref{\PA_n}(\sigma_i)|=
    \begin{cases}
      1 & \qquad\mbox{if } i=1 \mbox{ or } \sigma_{i-1}>\sigma_i  \\
     1+ \displaystyle\sum_{\sigma_k\in\beta_1}y_{\sigma_k} &\qquad\mbox{if } T(\sigma_i)=\beta_1\\
1+\displaystyle\sum_{k=1}^{i-1}y_{\sigma_k}&\qquad\mbox{if $m(i)$ does not exist}\\
      \displaystyle\sum_{\sigma_k\in \beta_{m(i)}\alpha_{m(i)-1}\beta_{m(i)-1}\cdots \alpha_1\beta_1\sigma_i}\hspace{-.5in}y_{\sigma_k}&\qquad\mbox{if $m(i)$ exists}      
    \end{cases}
\end{equation}
with \[m(i)=\min\left\{1\leq j\leq \ell: \displaystyle\sum_{\sigma_k\in \alpha_j}y_{\sigma_k}\geq y_{\sigma_{i}}\right\}.\]
 \end{theorem}
\begin{proof}
    This follows directly from Theorem \ref{thm:assortments} and the fact that cars' parking preferences are independent.
\end{proof}

Theorem \ref{theorem:parking assortments} immediately implies the following result.
\begin{corollary}\label{cor:count pa}
Fix $\y=(y_1,y_2,\ldots,y_n)\in\N^n$ and for any $\sigma\in\Sym_n$, let $|\Out_{\PA_n(\y)}^{-1}(\sigma)|$ be as given by Theorem~\ref{theorem:parking assortments}. Then
    \[|\PA_n(\y)|=\sum_{\sigma\in\mathfrak{S}_n}|\Out_{\PA_n(\y)}^{-1}(\sigma)|.\]
\end{corollary}

\begin{example} We conclude by applying the results in this section when $\sigma=4123$ and $y=(1,2,1,2)$.
    \begin{itemize}[leftmargin=0in]
        \item If $i=1,2$, by Theorem \ref{thm:assortments} case 1, we have $|\Pref{\PA_4}(\sigma_1)|=|\Pref{\PA_4}(\sigma_2)|=1$.
        \item If $i=3$, then $\sigma_3=2$.
        By Example \ref{ex:partitioning sigma}, we find that $m(3)=1$ 
        since $\alpha_1$ has length $2$ and car $2$ could potentially park there. As this is the only index such that $y_4=2\geq y_3=2$, we can use Theorem \ref{thm:assortments} case 3 to find that
        $|\Pref{\PA_4}(\sigma_3)|=1+(y_{2}-1)+\sum_{\sigma_k\in\beta_1}y_{\sigma_k}=1+(2-1)+1=1+1+1=3$.
        \item If $i=4$, then $\sigma_4=3$. By Example \ref{ex:partitioning sigma}, we can use Theorem \ref{thm:assortments} case 2 to find that $|\Pref{\PA_4}(\sigma_4)|=1+\sum_{\sigma_k\in \beta_1}y_{\sigma_k}=1+y_1+y_2=1+1+2=4$.
    \end{itemize}
    Theorem \ref{theorem:parking assortments} yields
    $|\Out^{-1}_{\PA_4(\y)}(4123)|=1\cdot 1\cdot 3\cdot 4= 12.$
    In Table \ref{tab:final}, we provide the cardinality of the sets $\Out^{-1}_{\PA_4((1,2,1,2))}(\sigma)$ for all  $\sigma\in\Sym_{4}$. From that data, we then use Corollary \ref{cor:count pa} to find that 
    $|\PA_4(\y)|=192$.
    \begin{table}[h!]
        \centering
        \begin{tabular}{|c|c|c|c|c|c|c|c|}\hline
             $\sigma$& $|\Out^{-1}_{\PA_4(\y)}(\sigma)|$&$\sigma$& $|\Out^{-1}_{\PA_4(\y)}(\sigma)|$&$\sigma$& $|\Out^{-1}_{\PA_4(\y)}(\sigma)|$&$\sigma$& $|\Out^{-1}_{\PA_4(\y)}(\sigma)|$ \\\hline
             $1234$ & 40 &
$2134$ & 20 &
$3124$ & 15 &
$4123$ & 12\\\hline
$1243$ & 8 &
$2143$ & 4 &
$3142$ & 3 &
$4132$ & 2\\\hline
$1324$ & 10 &
$2314$ & 15 &
$3214$ & 5 &
$4213$ & 4\\\hline
$1342$ & 6 &
$2341$ & 12 &
$3241$ & 4 &
$4231$ & 3\\\hline
$1423$ & 6 &
$2413$ & 6 &
$3412$ & 6 &
$4312$ & 3\\\hline
$1432$ & 2 &
$2431$ & 3 &
$3421$ & 2 &
$4321$ & 1\\\hline
        \end{tabular}
        \caption{Cardinalities of the sets $\Out_{\PA_4((1,2,1,2))}^{-1}(\sigma)$ for each $\sigma\in\Sym_4$.}
        \label{tab:final}
    \end{table}
\end{example}

\begin{remark}
Although our results are enumerative, they do in fact describe the set of  preferences of the cars. Moreover, we note that the preferences for car $i$ are  always bounded above by the sum of the lengths of the cars parked to the left of car $i$ plus one for the spot in which car $i$ parks.
\end{remark}

\end{document}